\documentclass[graybox]{svmult}

\usepackage{ulem} 
\definecolor{forestgreen}{rgb}{0.13, 0.55, 0.13} 

\usepackage{type1cm}        
\usepackage{makeidx}         
\usepackage{graphicx}        
\usepackage{multicol}        
\usepackage[bottom]{footmisc}

\usepackage{newtxtext}       %
\usepackage[varvw]{newtxmath}       

\makeindex             

\newcommand{\dist}{\mathsf{d}}

\newcommand{\meas}{\mathfrak{m}}

\newcommand{\res}{\mathop{\hbox{\vrule height 7pt width .5pt depth 0pt
\vrule height .5pt width 6pt depth 0pt}}\nolimits}

\newcommand{\intav}{{\mathop{\int\kern-10pt\rotatebox{0}{\textbf{--}}}}}

\newcommand{\di}{\mathop{}\!\mathrm{d}}
\newcommand{\haus}{\mathscr{H}}
\newcommand{\Ch}{{\sf Ch}}

\definecolor{blue}{rgb}{0,0,0}  
\definecolor{forestgreen}{rgb}{0,0,0}  

\DeclareMathOperator{\RCD}{RCD}

\begin{document}

\title*{Locally homogeneous RCD spaces}
\author{Shouhei Honda, Artem Nepechiy}
\institute{Shouhei Honda \at  Graduate School of Mathematical Sciences, The University of Tokyo, \email{shouhei@ms.u-tokyo.ac.jp} 
\\
Artem Nepechiy \at  Institute for Algebra and Geometry, KIT, \email{artem.nepechiy@kit.edu}
}
%
%
\maketitle

\abstract*{ The goal of this note is to demonstrate how existing results can be adapted to establish the following result: A locally metric measure homogeneous $\RCD (K,N)$ space is isometric to a smooth Riemannian manifold. }

\abstract{ The goal of this note is to demonstrate how existing results can be adapted to establish the following result: A locally metric measure homogeneous $\RCD (K,N)$ space is isometric to, after multiplying a positive constant to the reference measure, a smooth Riemannian manifold with the Riemannian volume measure.}


\section{Introduction}\label{sec1}
\subsection{Results}
Let us start by recalling the main result of \cite{LN} by Lebedeva and the second named author:
\begin{theorem}[{\cite[Main Theorem]{LN}}]\label{maintheorem}
    Let \((M, g_0)\) be a locally homogeneous \(C^0\)-Riemannian manifold. 
    Then \((M, \dist_{g_0})\) is isometric to a smooth Riemannian manifold as metric spaces. 
\end{theorem} 
In the above a \textit{\(C^0\)-Riemannian manifold} is a pair \((M, g_0)\) consisting of a \(C^1\)-manifold \(M\) together with a continuous Riemannian metric \(g_0\). The Riemannian metric \(g_0\) induces a canonical length structure, which in turn induces an intrinsic distance function \(\dist_{g_0}\) on \(M\).
\par 

We call a metric space $(X, \dist)$ \textit{locally (metric) homogeneous} if for any two points of $X$ there is a
local isometry taking one to the other (see Definition \ref{def: local homog} for the precise one). 
Note that the local homogeneity does not imply the global one in general because of considering compact surfaces of constant negative curvature.
A direct consequence of Theorem \ref{maintheorem} is:
\begin{corollary}[{\cite[Corollary]{LN}}]\label{Alexandrov}
Any locally homogeneous Alexandrov space of finite dimension is isometric to a smooth Riemannian manifold.
\end{corollary}

An \textit{Alexandrov space} is, by definition, a metric space with a lower bound of sectional curvature in a synthetic sense. In connection with Corollary \ref{Alexandrov}, it is natural to ask whether an analogous result is satisfied or not, in the case of Ricci curvature (instead of sectional curvature)?

The aim of this note is to provide a positive answer to this question. The precise statement is as follows. See section \ref{secpre} for the preliminary on $\RCD$ spaces. 

\begin{theorem}[Locally metric measure homogeneous RCD\((K,N)\) spaces are smooth]\label{LHRCD}
  Let \((X, \dist, \meas)\) be a locally metric measure homogeneous (see Defintion \ref{def: locally measure homogeneous}) $\RCD(K,N)$ space of essential dimension $n$ for some \(K \in \mathbb{R}\) and some \(N \in [1, \infty)\). Then $(X,\dist)$ is isometric to a smooth Riemannian manifold and $\meas$ coincides with the Hausdorff measure $\haus^n$ of dimension $n$, up to multiplying a positive constant to the measure, namely $\haus^n=c\meas$ for some $c>0$.
\end{theorem}

Theorem \ref{LHRCD} is justified along with analogous ideas in the proof of Theorem \ref{maintheorem}, combining with the recent developments on $\RCD$ spaces. It is worth mentioning that a priori Theorem \ref{maintheorem} does not imply  Theorem \ref{LHRCD} because it is unclear whether the metric structure of a locally metric measure homogeneous RCD\((K,N)\) space is isometric to a $C^0$-Riemannian manifold. 
A detailed inspection of the proof yields that one can circumvent the requirement for the space to be a $C^1$-manifold in order for the argument to go through.  

Finally let us provide an application of Theorem \ref{LHRCD}.
\begin{corollary}\label{finite}
For all $n \in \mathbb{N}, K \in \mathbb{R}, d, r, v \in (0, \infty)$ there exists $\epsilon=\epsilon(n, K, d, r, v)>0$ such that the following hold: Let $\mathcal{M}:=\mathcal{M}(n, K, d, r, v)$ be the set of all isometry classes of non-collapsed $\RCD(K, n)$ spaces $(X, \dist, \haus^n)$ with $\mathrm{diam} X \le d, \haus^n(X) \ge v$ and 
\begin{equation*}\label{alsmos}
    \dist_{\mathrm{pGH}}\left((B_s(x), x), (B_s(y), y)\right)\le \epsilon s,\quad \text{for all $x, y \in X$ and $0<s \le r$,}
    \end{equation*}
    where $\dist_{\mathrm{pGH}}$ denotes the pointed Gromov-Hausdorff distance and $B_s(x)$ denotes the open ball of radius $s$ centered at $x$.
    Then:
\begin{enumerate}
    \item  \textcolor{forestgreen}{A}ny non-collapsed $\RCD(K, n)$ space $(X, \dist, \haus^n)$ belonging to $\mathcal{M}$ is homeomorphic to a locally 
    homogeneous compact Riemannian manifold. Moreover if $(X, \dist)$ is smooth, then the homeomorphism can be improved to be a diffeomorphism\textcolor{forestgreen}{.}
    \item $\mathcal{M}$ has only finitely many members up to homeomorphism:
    \item \textcolor{forestgreen}{D}enoting by $\mathcal{M}_{C^{\infty}}$ the set of all smooth elements in $\mathcal{M}$, $\mathcal{M}_{C^{\infty}}$ has only finitely many members up to  diffeomorphism.
\end{enumerate}
\end{corollary}
We refer \cite{Huang} for similar results {\color{blue}and see \cite{Santos} for a related result}.
\subsection{How to prove the main result}
\subsubsection{Lie group structure}
In \cite[Lemma 3.5]{LN} the smooth structure and the Riemannian metric was used to establish that the  group of isometries of a ball ${B}_r(p)$ fixing the point $p$ is a  Lie group. We intend to prove this result directly for RCD\((K,N)\) spaces. In the context of RCD\((K,N)\) spaces, the statements in this direction are known. For example the isometry group of an RCD\((K,N)\) space is a Lie group (see \cite[Theorem A]{zbMATH07027835} and \cite[Corollary 1.2]{zbMATH06920925}). In general a ball ${B}_r(p)$ in an RCD\((K,N)\) space needs not to be an RCD\((K,N)\)\textcolor{forestgreen}{\sout{-}}space itself. Therefore the mentioned results do not imply that the isometry group of ${B}_r(p)$ is a Lie group. However, we will obtain the following localization result:

\begin{proposition}[Isometry group of a ball in an RCD\((K,N)\) space is a Lie group]\label{Lie localization}
    Let \((X, \dist, \mathfrak{m})\) be an $\RCD(K,N)$ space for some \(K \in \mathbb{R}\) and some \(N \in [1, \infty)\). Then for any $p \in X$ and all $r>0$ the groups 
    $$ \operatorname{Iso}(B_r(p))= \left \lbrace f \colon B_r(p) \rightarrow B_r(p) \vert \, \,   f \text{ isometry between metric spaces}  \right \rbrace, $$
    $$ \operatorname{Iso}_p(B_r(p))= \left \lbrace f \colon B_r(p) \rightarrow B_r(p) \vert \, \,  f(p)=p , f \text{ isometry between metric spaces}  \right \rbrace $$
    are  Lie groups.
\end{proposition}
{\color{blue}Note that the proof is essentially same to that of the global result (namely the ball is replaced by $X$) following \cite{CN}, together with results in \cite{deng2020holder}. We will provide a self-contained proof for readers' convenience.
}
\subsubsection{Constructing an isometry}\label{constiso}
{\color{blue}Firstly let us provide a brief summary of the proof of Theorem \ref{maintheorem}.}\par
{\color{blue} Observe that Theorem \ref{maintheorem} is a local statement. Therefore it is enough to construct locally an isometry from the $C^0$-manifold $M$ to a smooth Riemannian manifold. The smooth Riemannian manifold will appear via a homogeneous space construction, which will use the local isometries of $M$. \par 
If one fixes a point $p \in M$ and considers an appropriately chosen small neighborhood $U$ of $p$ one has the problem that the set of local isometries between subsets of $U$ is a priori not even a local group, since there is no meaningful way to define the composition of maps satisfying the axioms of a local group (for definitions see
\cite{spiro1993remark}, \cite{spiro1992lie}, \cite{mostow1950extensibility}, \cite{MR1088511}). From this set however, one can construct a local group $U_G$, if one would know that for $B_r(p),B_r(q) \subset U$ every local isometry $f \colon B_r(p)  \rightarrow B_r(q)$ could be extended to a ball of fixed radius $F \colon B_R(p) \rightarrow B_R(q)$ (independent of $p,q,r$). The key to proving such an extension property in \cite{LN} is to apply the Hilbert-Smith theorem and to conclude that $\operatorname{Iso}_p(B_r(p))$ is a Lie group.  \par 
Once one has established that $U_G$ is a local group, once can use the resolution of Hilbert's fifth Problem for local groups to conclude that $U_G$ must also be a local Lie group. \par 
For local groups there is a meaningful way of defining a quotient. If we take the quotient of the local Lie group $U_G$ by the stabilizer of $p$, which we denote by $U_H$,  we obtain a smooth manifold together with a homeomorphism from it to an open subset of $M$.\par 
It remains to show that this homeomorphism is indeed an isometry, if one endows the quotient $U_G/U_H$ with the correct metric. This happens the following way: One can construct some metric on $U_G/U_H$ such that the canonical homeomorphism becomes a Lipschitz map. Using Rademacher's theorem and the fact that we are dealing with smooth manifolds one gets that the canonical map is differentiable almost everywhere and at some point the differential is an isomorphism. Using this isomorphism between tangent spaces one can pull back the original metric at a point and define a left invariant metric on the quotient. The final step is to prove that the canonical homeomorphism is indeed an isometry if $U_G/U_H$ carries this left invariant metric.}

 {\color{blue}Let us turn to our setting.} Using Proposition \ref{Lie localization} the proof {\color{blue}of Theorem \ref{maintheorem}} in \cite{LN} {\color{blue}explained above} can be carried out verbatim 
 until the very last step right after \cite[Corollary 4.5]{LN}. At this point one has a homeomorphism $$h \colon B_r(p) \subset (X,\dist) \rightarrow U \subset  U_G/U_H$$ between neighborhoods around a point $p \in X$ and a neighborhood $U$ in a {\color{blue} smooth Riemannian manifold} $U_G/U_H$, on which $U_G$ acts by left multiplication $L_f$.\par
At this point one would need to construct a local isometry between the RCD\((K,N)\) space $(X,\dist,\mathfrak{m})$ and the local smooth model space $U_G/U_H$ with a yet to be determined left invariant Riemannian metric $g$. In \cite{LN} one could make $h$ into a Lipschitz map and pull back the continuous Riemannian metric onto $U_G/U_H$. In the case when $(X,\dist,\mathfrak{m})$ is an RCD$(K,N)$ space it is a priori not clear, whether the distance $\dist$ on $X$ is induced by some Riemannian metric. However, any RCD\((K,N)\) space admits a canonical almost everywhere defined Riemannian metric. Thus our task is to describe a way on how to construct a suitable pullback metric and to show that the induced metric makes our map $h$ into an isometry. 
{\color{blue}
 \subsection{Structure of the paper}
The paper is organized as follows. In the next section, Section \ref{secpre}, we recall fundamental notions/results on metric measure spaces, including a brief introduction on 
 $\RCD$ spaces. Section \ref{secprop} is devoted to the proof of Proposition \ref{Lie localization}. In Section \ref{secnon}, we show the implication from a local homogeneity to the non-collapsed condition, which allows us to use fine topological properties on non-collapsed $\RCD$ spaces. This is necessary to run similar arguments as done in \cite{LN}. We recall known results about Lebesgue points of certain functions on an $\RCD$ space in Section \ref{secleb}, which will be used to discuss blow-up behaviors of a map from an RCD space into a manifold in Seciton \ref{deriv}. Finally, in the final section, Section \ref{secfinal}, the proof of Theorem \ref{LHRCD} is completed.}
 \section{Preliminaries on metric measure spaces}\label{secpre}
In this section we collect the basic notions, definitions and known results about RCD$(K,N)$ spaces. We start with the definitions of local homogeneity. 
\subsection{Local homogeneity}
The following is a fundamental notion in the paper.


\begin{definition}[Local homogeneity as metric space]  \label{def: local homog}  
A metric space $(X, \dist)$ is called \textit{locally (metric) homogeneous} if for every $x,y \in X$ there exist $r>0$ and
	an isometry $f:B_r(x) \rightarrow B_r(y)$ satisfying $f(x)=y$. We call such a map $f$ a pointed isometry.
\end{definition}

We now introduce the notions necessary for RCD$(K,N)$ spaces and its local homogeneity. 

\begin{definition}[Metric measure space]\label{defmm} We call a triple $(X,\dist,\mathfrak{m})$ a \textit{metric measure space} if the pair $(X,\dist)$ is a complete separable metric space and the measure $\mathfrak{m}$ is a Borel measure on $X$ with full support and finite on each bounded set. 
\end{definition}

\begin{definition}[Locally homogeneity as metric measure space] \label{def: locally measure homogeneous}
We say that a metric measure space $(X, \dist, \meas)$ is \textit{locally metric measure homogeneous} if for all $x, y \in X$ there exist $r>0$ and an isometry $\Phi:B_r(x) \to B_r(y)$ as metric spaces with $\Phi(x)=y$ such that $\Phi_{\sharp}(\meas \res_{B_r(x)})=\meas \res_{B_r(y)}$ holds, where $\Phi_{\sharp}$ denotes the push-forward.
\end{definition}
For a metric measure space, the local metric measure homogeneity, by definition, implies the metric one, but it is trivial that the converse implication is not true in general. However the converse one is true if the reference measure is a Hausdorff measure, see for instance (\ref{liphaus}).


\subsection{$\RCD$ spaces}
Fix a metric measure space $(X, \dist, \meas)$ {\color{blue}(recall Definition \ref{defmm}).}
\begin{definition}[Sobolev space]
    The \textit{Cheeger energy} $\Ch:L^2(X, \meas) \to [0, \infty]$ is defined by
\begin{equation}\label{cheegert}
\Ch(f)=\inf_{\{f_i\}_i}\left\{ \liminf_{i\to \infty}\frac{1}{2}\int_X\left(\mathrm{Lip}f_i(x)\right)^2\di \meas(x) \right\},
\end{equation}
where the infimum $\{f_i\}_i$ runs over all bounded Lipschitz $L^2$-functions with  $\|f_i - f\|_{L^2} \to 0$ as $i \to \infty$, and $\mathrm{Lip}f(x)$ denotes the local slope of $f$ at $x$ defined by:
\begin{equation*}
\mathrm{Lip}f(x):=\limsup_{y \to x}\frac{|f(x)-f(y)|}{\dist(x, y)}.
\end{equation*}
The \textit{Sobolev space} $H^{1,2}=H^{1,2}(X, \dist, \meas)$ is defined by the finiteness domain of $\Ch$ and then it is a Banach space equipped with the norm $\|f\|_{H^{1,2}}:=(\|f\|_{L^2}^2+2\Ch(f))^{\frac{1}{2}}$. 
\end{definition}

For any $f \in H^{1,2}$, considering the collection $R(f)$ of all functions in $L^2(X, \meas)$ larger than or equal to a weak $L^2$-limit of $\mathrm{Lip}f_i(x)$ appeared in (\ref{cheegert}), we know that $R(f)$ is a closed convex non-empty subset in $L^2(X, \meas)$. Denoting by $|\nabla f|$ the unique element of $R(f)$ with the smallest $L^2$-norm, called the \textit{minimal relaxed slope} of $f$, we have

\begin{equation*}
\Ch(f)=\frac{1}{2}\int_X|\nabla f|^2\di \meas.
\end{equation*}
The following notion allows us to discuss a canonical Riemannian metric in a weak sense.
\begin{definition}[Infinitesimally Hilbertianity]
    $(X, \dist, \meas)$ is called \textit{infinitesimally Hilbertian} (IH) if the Sobolev space $H^{1,2}$ is a Hilbert space. If it is IH, then for all $f, h \in H^{1,2}$,
\begin{equation*}
\langle \nabla f, \nabla h\rangle := \lim_{t \to 0}\frac{|\nabla (f+th)|^2-|\nabla f|^2}{2t}
\end{equation*}
determines an $L^1$-function on $X$. 
\end{definition}
Then we can define the \textit{Laplacian} on $(X, \dist, \meas)$ from the point of view of Dirichlet form theory as follows.
\begin{definition}[Laplacian]
    Assume that $(X, \dist, \meas)$ is IH.
We denote by $D(\Delta)$ the domain of the Laplacian, namely, $f \in D(\Delta)$ holds if and only if $f \in H^{1,2}$ and there exists a unique $\phi \in L^2(X, \meas)$, denoted by $\Delta f$, such that
\begin{equation*}
\int_X\langle \nabla f, \nabla h\rangle \di \meas=-\int_X\phi h\di \meas,\quad \text{for any $h \in H^{1,2}$.}
\end{equation*}
\end{definition}
We are now in a good position to define $\RCD(K, N)$ spaces.
\begin{definition}[RCD space]
$(X, \dist, \meas)$ is said to be an \textit{$\RCD(K, N)$ space} (or $\RCD$ \textit{space}, for short) for some $K \in \mathbb{R}$ and some $N \in [1, \infty)$ if the following four conditions are satisfied:
\begin{enumerate}
\item \textcolor{forestgreen}{I}t is IH:
\item \textcolor{forestgreen}{T}here exist a positive constant $C>1$ and a point $x \in X$ such that 
\begin{equation*}
\meas(B_r(x))\le C\exp (Cr^2),\quad \text{for any $r>1$:}
\end{equation*}
\item \textcolor{forestgreen}{I}f a Sobolev function $f \in H^{1,2}$ satisfies $|\nabla f|(x) \le 1$ for $\meas${\color{blue}-}a.e. $x \in X$, then $f$ has a $1$-Lipschitz representative:
\item \textcolor{forestgreen}{W}e have the \textit{Bochner inequality}:
\begin{equation}\label{s8sabasy}
\frac{1}{2}\int_X|\nabla f|^2 \cdot \Delta \phi \di \meas \ge \int_X\phi \left( \frac{(\Delta f)^2}{N} +\langle \nabla \Delta f, \nabla f\rangle +K|\nabla f|^2 \right)\di \meas
\end{equation}
for all $\phi \in L^{\infty}(X, \meas) \cap D(\Delta)$ with $\phi \ge 0$ and $\Delta \phi \in L^{\infty}(X, \meas)$, and $f \in D(\Delta)$ with $\Delta f \in H^{1,2}$.
\end{enumerate}
\end{definition}
See \cite{A, STURM} for nice surveys about $\RCD$ spaces, including the history.

Let us end this subsection by {\color{blue}recalling} a special subclass of $\RCD(K, N)$ spaces introduced in \cite{DG}.
\begin{definition}[Non-collapsed $\RCD$ space]
An $\RCD(K, N)$ space $(X, \dist, \meas)$ for some $K \in \mathbb{R}$ and some $N \in [1, \infty)$ is said to be \textit{non-collapsed} if $\meas$ coincides with the Hausdorff measure $\haus^N$ of dimension $N$.
\end{definition}
{\color{blue}There exist fine properties for non-collapsed RCD spaces, including a fact that $N$ coincides with the essential dimension (thus it must be an integer), and the following \textit{Reifenberg flatness}. We refer \cite{DG, KM} for the proofs.
\begin{theorem}[Reifenberg flatness]\label{thm:reifenberg}
For all $K\in \mathbb{R}, N\in \mathbb{N}$ and $\epsilon>0$ there exist $\delta=\delta(K, N, \epsilon)>0$ such that if a non-collapsed $\RCD(K, N)$ space $(X, \dist, \haus^N)$ satisfies
$$\dist_{\mathrm{pGH}}((B_{r}(p) ,p), (B_{r}(0_N),0_N)) < \delta \cdot r$$
for some $0<r<1$ and some $p \in X$, where $0_N$ is the origin of $\mathbb{R}^N$,
then 
$$
    \dist_{\mathrm{pGH}}((B_{s}(p) ,p), (B_{s}(0_n),0_n)) < \epsilon \cdot s
    $$
for all $0<s \le \frac{r}{4}$ and $q \in B_{\frac{r}{2}}(p)$.
\end{theorem}
}

\subsection{Fine structure on $\RCD$ space}\label{subfine}
Let us fix an $\RCD(K, N)$ space $(X, \dist, \meas)$ for some $K \in \mathbb{R}$ and some $N \in [1, \infty)$. {\color{blue}It follows from \cite[Corollary 2.4]{zbMATH05049052} that $X$ is proper, namely any bounded closed subset is compact. Further fine structures are described as follows.}
\begin{theorem}[{\cite[Corollary 1.4]{zbMATH06658550}}]\label{Theorem: RCD is intrinsic} $(X, \dist)$ is a geodesic space. Moreover
    for every \(x \in X\) the following holds: for \(\meas \)-a.e. \(y \in X\) there is only one geodesic connecting \(y\) to \(x\).

\end{theorem}
In the sequel, we assume that $X$ is not a single point.
\begin{definition}[Tangent cone]
A pointed metric measure space $(Y, \dist_Y, \meas_Y, y)$ is said to be a \textit{tangent cone} at $x \in X$ of $(X, \dist, \meas)$ if there exists a sequence $r_i \to 0^+$ such that 
\begin{equation*}
\left(X, \frac{1}{r_i}\dist, \frac{1}{\meas(B_{r_i}(x))}\meas, x\right) \stackrel{\mathrm{pmGH}}{\to} (Y, \dist_Y, \meas_Y, y),
\end{equation*}
where ``pmGH'' stands for the pointed measured Gromov-Hausdorff convergence. Denote by $\text{Tan}(X, \dist, \meas, x)$ the set of all isometry classes of tangent cones at $x$ of $(X, \dist, \meas)$. 
\end{definition}
Let us recall a fundamental notion to study $\RCD$ spaces, \textit{regular sets}.
\begin{definition}[Regular set]
A point $x \in X$ is said to be \textit{$k$-dimensional regular} if 
\begin{equation*}
\text{Tan}(X, \dist, \meas, x) =\left\{ \left(\mathbb{R}^k, \dist_{\mathrm{Euc}}, \frac{1}{\omega_k}\haus^k, 0_k\right)\right\},
\end{equation*}
where $\dist_{\mathrm{Euc}}$ denotes the standard Euclidean distance, and $\omega_k$ is the volume of a ball of radius $1$ in $\mathbb{R}^k$.
Denote by $\mathcal{R}_k$ the set of all $k$-dimensional regular points of $(X, \dist, \meas)$.
\end{definition}
We will also use the following results (see also \cite[Theorem 6.2]{deng2020holder}).
\begin{theorem}[{\cite[Theorem 0.1]{BrueSemola}}]\label{Deng theorem regular set}
    There exists a unique \(n \in \mathbb{N}\) with  \(1 \leq n \leq N\) such that \(\meas(X \setminus \mathcal{R}_n) = 0\).
We call $n$ the \textit{essential dimension} of $(X, \dist, \meas)$.
\end{theorem}
In the sequel, denote by $n$ the essential dimension of $(X, \dist, \meas)$.
Combining Theorem \ref{Deng theorem regular set} with (the proof of) the rectifiability result in \cite[Theorem 1.1]{MondinoNaber} and \cite[Propositions 3.7 and 3.9]{BPS0} (see also \cite{BPS2}), we obtain the following (see also \cite[Theorem 4.1]{AmbrosioHondaTewodrose}).
\begin{theorem}\label{reduced}
    We have the following.
    \begin{enumerate}
        \item Let us denote by $\mathcal{R}_n^*$ the set of all points $x \in \mathcal{R}_n$ satisfying that the finite positive limit 
        \begin{equation}\label{RNRN}
            \lim_{r\to 0^+}\frac{\meas(B_r(x))}{\omega_nr^n} \in (0, \infty)
        \end{equation}
        exists. Then 
        \begin{enumerate}
            \item we have $\meas(X\setminus \mathcal{R}_n^*)=0$:
            \item $\meas\res_{\mathcal{R}_n^*}$ and $\haus^n\res_{\mathcal{R}_n^*}$ are mutually absolutely continuous:
            \item the limit in (\ref{RNRN}) coincides with the Radon-Nikodym derivative: $$\frac{\di \meas\res_{\mathcal{R}_n^*}}{\di \haus^n\res_{\mathcal{R}_n^*}}(x)$$ for $\meas$-a.e. $x \in \mathcal{R}_n^*$.
        \end{enumerate}
        \item For $\meas$-a.e. $x \in \mathcal{R}_n^*$, it holds that for any $\epsilon>0$ there exist $r>0$ and a harmonic map 
        \begin{equation*}
        \Phi=(\phi_1, \phi_2,\ldots, \phi_n):B_r(x) \to \mathbb{R}^n
        \end{equation*}
        such that 
        \begin{equation*}
        \intav_{B_r(x)}\left|\langle \nabla \phi_i, \nabla \phi_j\rangle-\delta_{ij}\right|\di \meas <\epsilon,\quad \text{for all $i, j$,}
        \end{equation*}
        where $\intav_A:=\frac{1}{\meas(A)}\int_A$.
        \item For any $\epsilon>0$ there exists a family $\{(A_i, \Phi_i)\}_{i \in \mathbb{N}}$ of $(1\pm \epsilon)$-bi-Lipschitz embeddings $$\Phi_i:A_i \hookrightarrow \mathbb{R}^n$$ of Borel subsets $A_i$ of $X$ such that each $\Phi_i$ can be obtained as a restriction of a harmonic map defined on a neighborhood of $A_i$ and that $\meas (X\setminus \bigcup_iA_i)=0$.
    \end{enumerate}
\end{theorem}
See also \cite[Proposition 4.3]{BrueMondinoSemola}.

In order to introduce a geometric property on $\mathcal{R}_n$, Theorem \ref{Deng theorem weakly connected},  let us recall the following notions.
\begin{definition}[\(\epsilon\)-geodesic]
     A unit speed, piecewise geodesic curve \(\gamma\) between \(p \in X\) and \(q\in X\) is called an \(\epsilon\)-\textit{geodesic}  if \(L(\gamma) - \dist(p,q) \leq \frac{\epsilon}{2}\dist (p,q)\), where $L(\gamma)$ denotes the length of \(\gamma\).
\end{definition}
\begin{definition}[Weak convexity]
A subset \(S \subset X\) is called \textit{weakly convex} if  for all \((x, y) \in S \times S\) and \(\epsilon > 0\), there exists an \(\epsilon\)-geodesic \(\gamma \subset S\) connecting \(x\) and \(y\).
\end{definition}
The following result will play a key role to prove Proposition \ref{Lie localization}.
\begin{theorem}[{\cite[Theorem 6.5]{deng2020holder}}] \label{Deng theorem weakly connected}
    \(\mathcal{R}_n\) is weakly convex. In particular, \(\mathcal{R}_n\) is path-connected.
\end{theorem}
\subsection{Functional analysis with respect to pmGH convergence}
The following stability result for harmonic functions will play a key role to construct an isometry explained in subsection \ref{constiso}. 
See \cite[Theorem 4.4 or Corollary 4.5]{AmbrosioHonda} for the proof and see \cite[Theorem 3.1]{Jiang} for the Lipschitz regularity of harmonic functions.
\begin{theorem}[Stability of harmonic functions]\label{harmonic stab}
Let
\begin{equation*}\label{pmghrcd}
\left( X_i, \dist_i, \meas_i, x_i\right) \stackrel{\mathrm{pmGH}}{\to} \left( X, \dist, \meas, x\right).
\end{equation*}
be a pmGH convergent sequece of pointed $\RCD(K, N)$ spaces for some $K \in \mathbb{R}$ and some $N \in [1, \infty)$.
Take $R>0$ and let $f_i:B_R(x_i) \to \mathbb{R}$ be a harmonic function for any $i<\infty$ with $\sup_i\|f_i\|_{L^1(B_R(x_i))} <\infty$. Then after passing to a subsequence there exists a harmonic function $f:B_R(x) \to \mathbb{R}$ such that $f_i$ locally uniformly converge to $f$ on $B_R(x)$ and that they also $H^{1,2}$-strongly converge to $f$ on $B_r(x)$ for any $r<R$.
\end{theorem}
See \cite{GigliMondinoSavare13, AmbrosioHonda0, AmbrosioHonda} for the definition of $H^{1,2}$-strong convergence and more general results (see also \cite{GIGLI} for a nice survey).

{\color{blue}In connection with the theorem above,} let us end this subsection by {\color{blue}recalling a general result on} the harmonicity of any blow-up limit {\color{blue}(or a kind of differentiability)} of a Sobolev function {\color{blue}as an independent interest (strictly speaking, together with results in Section \ref{secleb}, one of Theorems \ref{harmonic stab}, \ref{Rad} is enough for our purpose. However writing both statements may be useful for readers' understanding).} The proof is based on \cite[Theorem 3.7]{Cheeger}. See \cite[Theorem 5.4]{AHPT} for the details.
\begin{theorem}[Rademacher type differentiability]\label{Rad}
Let $(X, \dist, \meas)$ be an $\RCD(K, N)$ space for some $K\in \mathbb{R}$ and some $N \in [1, \infty)$, let $R>0$, let $x \in X$ and
let $f:B_R(x) \to \mathbb{R}$ be a $H^{1,2}$-Sobolev function. Then for $\meas$-a.e. $z \in B_R(x)\cap \mathcal{R}_n$ we have the following: for any $r>0$, let
\begin{equation}\label{norma}
f_{r, z}:=\frac{1}{r}\left( f-\intav_{B_r(z)}f\di \meas\right).
\end{equation}
Then for any convergent sequence $r_i \to 0^+$, after passing to a subsequence, $f_{r_i, z}$ $H^{1,2}_{\mathrm{loc}}$-strongly converge to a linear function on $\mathbb{R}^n$ with respect to the pmGH convergence:
\begin{equation*}
\left(X, \frac{1}{r_i}\dist, \frac{1}{\meas(B_{r_i}(z))}\meas, z\right) \stackrel{\mathrm{pmGH}}{\to} \left(\mathbb{R}^n, \dist_{\mathrm{Euc}}, \frac{1}{\omega_n}\haus^n, 0_n\right).
\end{equation*}
\end{theorem}
\begin{remark}\label{Radrem}
In Theorem \ref{Rad}, if $f$ is locally Lipschitz, then the normalization (\ref{norma}) can be replaced by
\begin{equation*}
    \frac{1}{r}\left(f-f(z)\right)
\end{equation*}
in order to get the same conclusion.
\end{remark}
\section{Proof of Proposition \ref{Lie localization}}\label{secprop} 
In this section let us prove Proposition \ref{Lie localization}. {\color{blue}Note that this general statement is a direct consequence of the proof of \cite[Theorem 1.14]{CN} together with that of \cite[Theorem 6.2]{deng2020holder} (see the end of this section).
Since it is enough for our purpose to prove the proposition for non-collapsed $\RCD$ spaces because of Lemma \ref{top}, firstly let us provide a self-contained proof of it specializing in the non-collapsed case for readers' convenience.} For this, we slightly modify the proof of \cite[Theorem 4.5]{zbMATH01782634} with the use of Theorems \ref{Deng theorem regular set} and \ref{Deng theorem weakly connected}. Let us first recover the statements in \cite[Theorem A]{zbMATH07027835} and \cite[Corollary 1.2]{zbMATH06920925}. For the readability we restate the result.

\begin{theorem}[Isometry group of RCD\((K,N)\) space is a Lie group]
      Let \((X, \dist, \mathfrak{m})\) be an $\RCD(K,N)$ space for some \(K \in \mathbb{R}\) and some \(N \in [1, \infty)\). Then the isometry group $\operatorname{Isom}(X)$ of $(X, \dist)$ is a Lie group.
\end{theorem}
\begin{proof}[Proof {\color{blue}under the non-collapsed condition}]
Firstly let us introduce some notation: 
for a metric space $(Z, \dist)$ and a closed subgroup $H < \operatorname{Isom}(Z)$ of the isometry group, set 
$$ \rho_H(z):= \sup_{h \in H } \dist(z,h(z)), \qquad \qquad D_{H,r}(z) := \sup_{w \in B_{\frac{r}{2}}(z)} \rho_H(w).$$

    In order to prove that $\operatorname{Isom}(X)$ is a Lie group, it is sufficient to show that it contains no small closed subgroups \cite[Corollary 1.5.8]{TAO}. Assume {\color{blue}to} the contrary, that is, there exits a sequence $(H_i)_{i \in \mathbb{N}}$ of non-trivial closed subgroups, which satisfy 
    \begin{equation}\label{D0}
    ´\lim_{i \rightarrow \infty} D_{H_i,r}(p)=0 \qquad \text{ for all } r>0 \text{ and } p \in X.
    \end{equation}
    Denoting by $n$ the essential dimension of $(X, \dist, \meas)$,  we can  obtain a contradiction if we can find sequences of points $(p_i)_{i \in \mathbb{N}} \subset X$ and of positive numbers $r_i \rightarrow 0^+ $ such that $$ D_{H_i,r_i}(p_i)=\frac{1}{20}r_i, \qquad \dist_{\mathrm{pGH}}((B_{r_i}(p_i) ,p_i), (B_{r_i}(0_n),0_n)) < \epsilon_i \cdot r_i \text{ with } \epsilon_i \rightarrow 0,$$
    because of the following reason: thanks to \cite[Proposiion 3.6]{FY}, after taking a convergent subsequence we have 
        $$ \left( {\color{blue}B_{r_i}(p_i)},\frac{1}{r_i}\cdot \dist,p_i, H_i\right)  \stackrel{\mathrm{peGH}}{\rightarrow} \left ({\color{blue}B_1(0_n)},\dist_{\mathrm{Euc}},0_n,H \right),$$
    for some {\color{blue}non-trivial} closed subgroup $H < \operatorname{Isom}(\mathbb{R}^n)$, where ``peGH'' stands for the pointed equivalent Gromov-Hausdorff convergence. We have since the quantity $D_{H,1}$ is defined in purely metric terms and thus carries over by {\color{blue}pe}GH-convergence (with the rescaled metric):
    \begin{equation}\label{lower}
        {\color{blue}\frac{1}{20}=\liminf_{i \to \infty}D_{H_i,1}(p_i)\ge D_{H,1}(0_n).}
    \end{equation} But it is well known that there exists no non-trivial closed subgroup  $H<\operatorname{Isom}(\mathbb{R}^n)$ with $D_{H,1}(0_n)\leq \frac{1}{20}.$
    \par
    Therefore, {\color{blue}in the sequel, let us} find $p_i$ and $r_i$ as above. {\color{blue}Fix} an arbitrary $p \in \mathcal{R}_n$. 
    
   {\color{blue}Firstly it follows from (\ref{D0}) after taking a suitable subsequence that there exists $s_i \to 0^+$ such that 
    $$D_{H_i, s_i}(p) \le \frac{s_i}{100}.$$
    Since $H_i$ is non-trivial, we can find $q_i \in X$ with $\rho_{H_i}(q_i)>0$. Recalling that the set $\mathcal{R}_n$ is dense by Theorem \ref{Deng theorem regular set}, with no loss of generality we can assume $q_i \in \mathcal{R}_n$.

    Fix a sequence $\epsilon_i \rightarrow 0$ and consider the function 
        $$\mathrm{GHR}_{\epsilon_i} \colon \mathcal{R}_n \rightarrow \mathbb{R}, q \mapsto \sup\lbrace r>0 \vert \, \dist_{\mathrm{pGH}}((B_{r}(q) ,q), (B_{r}(0_n),0_n)) < \epsilon_i \cdot r  \rbrace.$$
        This function is lower semicontinuous on $\mathcal{R}_n$. Therefore it takes its minimum $>0$ on every continuous path $\gamma_i$ connecting $p$ and $q_i$, which exists since $\mathcal{R}_n$ is path connected by Theorem \ref{Deng theorem weakly connected}. 
        Denoting by $\bar s_i>0$ this minimum on a fixed $\gamma_i$, then we have for some $\delta_i \to 0^+$ $$\dist_{\mathrm{pGH}}((B_{s}(\gamma_i(t)) ,\gamma_i(t)), (B_{s}(0_n),0_n)) < \delta_i \cdot s$$ for all $t$ and $0<s<\bar s_i$, 
        where we used the non-collapsed condition here with Theorem \ref{thm:reifenberg}.
    
    From now on, fixing $0<t_i<\min\{s_i, \bar s_i\}$ with $D_{H_i, t_i}(q_i) \ge \frac{t_i}{20}$ whose existence is guaranteed by the observation (\ref{lower}), let us divide the arguments into the following two cases.
    Note that the function $D_{H_i,r}(q)$ for fixed $H_i$ is continuous in $q,r$ if the closure of an open ball coincides with the closed ball. This condition is true in inner metric spaces. However by completeness of $X$ and Theorem \ref{Theorem: RCD is intrinsic} finite dimensional RCD$(K,N)$ spaces are inner.

    \textit{The case when $D_{H_i, t_i}(p) \ge \frac{t_i}{20}$.} Applying the intermediate value theorem to the function $s \mapsto s^{-1}D_{H_i, s}(p)$, we can find $t_i \le u_i \le s_i$ with $D_{H, u_i}(p)=\frac{u_i}{20}$.
    
    \textit{The case when $D_{H_i, t_i}(p) \le \frac{t_i}{20}$.}  Recalling $D_{H_i, t_i}(q_i) \ge \frac{t_i}{20}$, we can also apply the intermediate value theorem to the function $D_{H_i, t_i}(\gamma_i(t))$ in $t$, in order to show $D_{H_i, t_i}(\gamma_i(v_i))=\frac{t_i}{20}$ for some $v_i$.


    In the both cases above, we can find the desired $p_i, r_i$. Thus we conclude. 
    }

    \end{proof}

With a minor modification we obtain Proposition \ref{Lie localization} as a corollary. 

\begin{proof}[Proof of Proposition \ref{Lie localization} {\color{blue}under the non-collapsed condition}]  
We can apply the same proof as above with the following modification: Assume $\operatorname{Iso}(B_r(p))$ is not a Lie group. As above we construct sequences $H_i$, $p_i \in B_r(p)$ and $r_i>0$. The set of regular points $\mathcal{R}_k \cap B_{2r}(p)$ is still dense  in $B_{2r}(p)$ and by Theorem \ref{Deng theorem weakly connected} the set $\mathcal{R}_k$ is weakly convex. By this, if we fix $\varepsilon>0$ small enough we can assume that all paths $\gamma_i$ are contained in $B_{2r}(p)$ and arrive at the same contradiction, {\color{blue}where note that any open ball of an inner metric space is also inner.}
\end{proof}
{\color{blue}
In the general case (that is if the space is not non-collapsed) we can not make use of Theorem \ref{thm:reifenberg}. However, we can get around that using the results from \cite{deng2020holder}. 
\begin{proof}[Proof of Proposition \ref{Lie localization} in the general case]  
First observe that the Reifenberg flatness holds on a compact sub-interval of a geodesic by \cite[Theorem 1.1]{deng2020holder}. The second observation is that the set of points $(x,y) \in X \times X$ such that geodesics from $x$ to $y$ are extendable past $x$ and $y$ respectively has full measure (compare \cite[Proof of Theorem 6.2]{deng2020holder}). With this one can assume that the points in the proof under the non-collapsed condition lie in the interior of a geodesic and thus can apply the Reifenberg flatness due to \cite[Theorem 1.1]{deng2020holder}.

\end{proof}}

\section{Locally homogeneity implies the non-collapsed condition}\label{secnon}
Let $(X, \dist, \meas)$ be an $\RCD(K, N)$ space for some $K \in \mathbb{R}$ and some $N \in [1, \infty)$ and let us denote by $n$ the essential dimension in the sense of Theorem \ref{Deng theorem regular set}.

In the sequel, we assume that $(X, \dist, \meas)$ is locally metric measure homogeneous. The goal of this section to prove a partial result, Lemma \ref{top}, to prove Theorem \ref{LHRCD}. 
\begin{lemma}\label{meashaus}
We see that $\meas=c\haus^n$ holds for some $c>0$ and that any point of $X$ is an $n$-dimensional regular point.
\end{lemma}
\begin{proof}
Applying the local homogeneity, we have $X = \mathcal{R}_n^*$ (recall Theorem \ref{reduced} for the definition of $\mathcal{R}_n^*$). Moreover the limit  
\begin{equation}\label{RN}
\lim_{r \to 0^+}\frac{\meas(B_r(x))}{\omega_nr^n}
\end{equation} 
 does not depend on $x \in X$, thus we denote by $c$ the limit. 
 Then it follows from Theorem \ref{reduced} that $\meas=c\haus^n$ holds.
\end{proof}
Next let us check the non-collapsed condition in the sense of \cite{DG}.
\begin{lemma}\label{top}
$(X, \dist, \haus^n)$ is a non-collapsed $\RCD(K, n)$ space. In particular $X$ is bi-H\"older homeomorphic to a Riemannian manifold of dimension $n$.
\end{lemma}
\begin{proof}
Thanks to Lemma \ref{meashaus} and \cite[Theorem 1.5]{BGHZ} (after \cite[Corollary 1.3]{H}) {\color{blue}together with \cite[Theorem 3.12]{BrueSemola}}, it is enough to prove that for any compact subset $A \subset X$ there exist $r>0$ and $C>1$ such that 
\begin{equation}
C^{-1}\le \frac{\haus^n(B_s(x))}{s^n} \le C, \quad \text{for all $x \in A$ and $0<s \le r$.}
\end{equation}
However this is an easy consequence of the compactness of $A$ with the proof of Lemma \ref{meashaus} {\color{blue}because of \cite[Lemma 3.10]{LN}}. Thus we conclude the first statement. The last statement is a direct consequence of {\color{blue}applying Theorem \ref{thm:reifenberg} with the first statement, Lemma \ref{meashaus} and} \cite[Theorem A.1.1]{CheegerColding1} (see also \cite[Theorem 3.1]{KM}).
\end{proof}

\begin{remark}\label{remtop}
One needs to make sure that the topological assumptions are satisfied to run the same arguments as in \cite{LN}, compare Lemma 3.16 therein. However, these conditions are satisfied in view of Lemma \ref{top} and the existence of $\varepsilon$-GH approximations, which can be taken to be homeomorphisms.
\end{remark}

\section{Lebesgue point}\label{secleb}
In this section we prepare auxiliary results, related to the subharmonicity, which are already observed in \cite[Remark 2.10]{BPS}. Though they are avoidable to realize Theorem \ref{LHRCD}, they give us better understanding in the discussions below. Note that $D(\Delta, B_R(x))$ denotes the domain of (local) Laplacian defined on a ball $B_R(x)$ of $X$, see for instance \cite[Definition {\color{blue}2.16}]{AmbrosioHonda} for the precise definition.
\begin{lemma}\label{0}
Let $(X, \dist, \meas)$ be an $\RCD(K, N)$ space and let $f \in D(\Delta, B_R(x))$. If $\Delta f$ is locally Lipschitz on $B_R(x)$, then the limit
\begin{equation*}
\lim_{r \to 0^+}\intav_{B_r(y)}|\nabla f|^2\di \meas
\end{equation*}
exists for any $y \in B_R(x)$. Moreover, denoting by $|\nabla f|^*(y)$ (or $|\nabla f|(y)$ for short) the square root of the limit, we have for all $y \in B_R(x)$ and $ p \in (0, \infty)$
\begin{equation*}
\lim_{r \to 0^+}\intav_{B_r(y)}\left||\nabla f|^p-|\nabla f|^*(y)^p\right|\di \meas=0.
\end{equation*}
\end{lemma}
\begin{proof}
The Bochner inequality (\ref{s8sabasy}) with the Lipschitz regularity result obtained in \cite[Theorem 3.1]{Jiang} shows
\begin{equation*}
\mathbf{\Delta} |\nabla f|^2\ge -C
\end{equation*}
for some $C>0$, where the Laplacian in the left-hand-side of the above is taken as a measure (cf. \cite{Gigli}).  Then finding $\phi$ with $\Delta \phi=C$ on a ball contained in $B_R(x)$, we know that $|\nabla f|^2+\phi$ is subharmonic on the ball. Then applying \cite[Proposition 8.24]{BjornBjorn} for $|\nabla \phi|^2+\phi$ (with the Lipschitz continuity of $\phi$ by \cite[Theorem 3.1]{Jiang}) completes the proof.
\end{proof}
\begin{corollary}\label{corcor}
Let $(X, \dist, \meas)$ be an $\RCD(K, N)$ space and let $f_i \in D(\Delta, B_R(x)) (i=1,2)$. If each $\Delta f_i$ is locally Lipschitz on $B_R(x)$, then the limit
\begin{equation*}
\lim_{r \to 0^+}\intav_{B_r(y)}\langle \nabla f_1, \nabla f_2\rangle \di \meas
\end{equation*}
exists for any $y \in B_R(x)$. Moreover, denoting by $\langle \nabla f_1,\nabla f_2\rangle^*(y)$ (or $\langle \nabla f_1, \nabla f_2\rangle(y)$ for short) the limit, we have for any $y \in B_R(x)$
\begin{equation*}
\lim_{r \to 0^+}\intav_{B_r(y)}\left|\langle \nabla f_1,\nabla f_2\rangle-\langle \nabla f_1, \nabla f_2\rangle^*(y)\right|\di \meas=0.
\end{equation*}
\end{corollary}
\begin{proof}
It is a direct consequence of Lemma \ref{0} with the following identity in $L^1$:
\begin{equation*}
\langle \nabla f_1, \nabla f_2\rangle= \frac{1}{2}\left(|\nabla (f_1+f_2)|^2-|\nabla f_1|^2-|\nabla f_2|^2 \right).
\end{equation*}
\end{proof}
\section{Derivative}\label{deriv}
Let $(X, \dist, \meas)$ be an $\RCD(K, N)$ space whose essential dimension is equal to $n$, let $M$ be a smooth (not necessarily Riemannian) manifold of dimension $n$ and let $h:X \to M$ be a map.
Assume that $h$ is locally Lipschitz with respect to some (thus equivalently any) Riemannian metric on $M$. 

The main purpose of this section is to show that the derivative
\begin{equation*}
(Dh)_x:\mathbb{R}^n \to T_xM
\end{equation*}
is canonically well-defined in some sense for $\meas$-a.e. $x \in \mathcal{R}_n$. Moreover if $(Dh)_x$ is bijective for such a point, $x \in \mathcal{R}_n$, then we will see that the push-forward scalar product 
\begin{equation*}
g^h_x=((Dh)_x)_{\sharp}g
\end{equation*}
on $T_{h(x)}M$ makes sense.

Before starting the precise discussion, we define the following.
\begin{definition}[Differentiability]\label{differen}
Let $A$ be a Borel subset of $\mathbb{R}^n$ with $\haus^n(A) >0$, let $x \in \mathrm{Leb}(A)$, namely $x \in A$ with $$\frac{\haus^n(B_r(x)\cap A)}{\haus^n(B_r(x))}\to 1$$ as $r \to 0^+$, and let $F:A \to \mathbb{R}^k$ be a Lipschitz map.
We say that $F$ is \textit{differentiable} at $x$ if there exists a Lipschitz map $\tilde F:\mathbb{R}^n \to \mathbb{R}^l$ such that $\tilde F|_A=F$ and that $\tilde F$ is differentiable at $x$.
\end{definition}
Note that this definition and the Jacobi matrix $J(\tilde F)_x$ of $\tilde F$ at $x$ do not depend on the choice of $\tilde F$ (thus we shall use the notation $J(F)_x=J(\tilde F)_x$).
\subsection{The case when $M$ is an open subset $U$ of $\mathbb{R}^n$; $M=U$.}
For any $x \in X$ and any small $r>0$, let 
\begin{equation*}\label{1}
h_r:=\frac{h-h(x)}{r}.
\end{equation*}
\begin{lemma}\label{pointwise}
For $\meas$-a.e. $x \in \mathcal{R}_n$, we have the following:  
under the pmGH-convergenve as $r \to 0^+$:
\begin{equation}\label{pmgh}
\left(X, \frac{1}{r}\dist, \haus^n_{\frac{1}{r}\dist}, x\right) \to \left(\mathbb{R}^n, \dist_{\mathrm{Euc}}, \haus^n, 0_n\right),
\end{equation}
for all two blow-up limit maps:
\begin{equation*}
h_0=\lim_{r_i \to 0^+}h_{r_i}: \mathbb{R}^n \to \mathbb{R}^n
\end{equation*}
and
\begin{equation*}
\tilde h_0=\lim_{s_i \to 0^+}h_{s_i}: \mathbb{R}^n \to \mathbb{R}^n,
\end{equation*}
we see that both $h_0$ and $\tilde h_0$ are linear maps and that they coincide, up to an action by an element of $O(n)$ in the domain.
In particular if some $h_0$ is bijective, then we can define the push-forward scalar product $g_{h(x)}=\langle \cdot, \cdot \rangle_{h(x)}$ on $T_{h(x)}M$  by
\begin{equation*}
\langle v, w \rangle_{h(x)}:=h_0^{-1}(v)\cdot h_0^{-1}(w)
\end{equation*}
which is independent of the choice of $r_i \to 0^+$.
\end{lemma}
\begin{proof}
By Theorem \ref{reduced} {\color{blue}(with the use of the maximal function)}, for $\meas$-a.e. $x \in \mathcal{R}_n$, we see that for any small $\epsilon>0$ there exist $r>0$ and a harmonic map 
\begin{equation*}
\Phi=(\phi_1, \ldots, \phi_n): B_{2r}(x) \to \mathbb{R}^n
\end{equation*}
with $\Phi(x)=0_n$ such that 
\begin{equation*}
\intav_{B_{s}(x)}\left|\langle \nabla \phi_i, \nabla \phi_j\rangle -\delta_{ij}\right|\di \meas <\epsilon.
\end{equation*}
holds for all $0<s\le r$ and $1 \le i \le j \le n$. Thus after multiplying a matrix of size $n$ (which is close to the identity matrix $E_n$) to $\Phi$, with no loss of generality, we can assume, by Corollary \ref{corcor}, that
\begin{equation*}
\lim_{s \to 0^+}\intav_{B_{s}(x)}\left|\langle \nabla \phi_i, \nabla \phi_j\rangle -\delta_{ij}\right|\di \meas=0
\end{equation*}
{\color{blue}(see also again \cite[Proposition 4.3]{BrueMondinoSemola}).}
Fix $x$ and $\Phi$ as above. Denote by 
\begin{equation*}
\Phi_0=(\phi_{0,1},\ldots, \phi_{0,n}):\mathbb{R}^n \to \mathbb{R}^n,\quad \tilde \Phi_0=(\tilde \phi_{0,1},\ldots, \tilde \phi_{0,n}):\mathbb{R}^n \to \mathbb{R}^n
\end{equation*}
the blow-up limits of $\frac{\Phi}{r_i}, \frac{\Phi}{s_i}$, respectively, with respect to (\ref{pmgh}) (after passing to a subsequence if necessary).
Then it follows from Theorem \ref{harmonic stab} that both $\Phi_0$ and $\tilde \Phi_0$ are isometries fixing the origin $0_n$.

Thanks to Theorem \ref{Rad} with Remark \ref{Radrem}, with no loss of generality we can also assume that both $h_0$ and $\tilde h_0$ are linear.
Thus denoting by $h=(h_1,\ldots, h_n)$,
\begin{equation*}
h_0=\left(\sum_{i=1}^na_{1, i}\phi_{0,i},\ldots, \sum_{i=1}^na_{n,i}\phi_{0, i}\right),\quad \tilde h_0=\left(\sum_{i=1}^n\tilde a_{1, i}\tilde \phi_{0,i},\ldots, \sum_{i=1}^n\tilde a_{n,i}\tilde \phi_{0, i}\right),
\end{equation*}
we have {\color{blue}by Corollary \ref{corcor}}
\begin{equation*}
a_{l, i}=\lim_{s \to 0^+}\intav_{B_s(x)}\langle \nabla h_l, \nabla \phi_i\rangle \di \meas=\tilde a_{l,i}.
\end{equation*}
The above observation allows us to conclude.
\end{proof}
\begin{lemma}\label{asyntt}
Under the same setting as in Lemma \ref{pointwise}, for any $C^2$-map $F=(f_1,\ldots, f_n):\mathbb{R}^n \to \mathbb{R}^n$, we have
\begin{equation*}
(F\circ h)_0=(DF)_{h(x)} \circ h_0,
\end{equation*}
where $D(F)_y$ denotes the derivative of $F$ at $y$. 
\end{lemma}
\begin{proof}
Recalling
\begin{equation*}
F(w)=F(z)+(w-z)J(F)_z+o(|w-z|),
\end{equation*}
we have
\begin{align}\label{tay}
\frac{F\circ h(y)-F\circ h(x)}{r}&=\frac{h(y)-h(x)}{r}\cdot J(F)_{h(x)}+\frac{1}{r}\cdot o(|h(y)-h(x)|) \nonumber \\
&=\frac{h(y)-h(x)}{r}\cdot J(F)_{h(x)}+\frac{1}{r}\cdot o(|y-x|).
\end{align}
Therefore taking the limit $r \to 0$ in (\ref{tay}), 
we conclude.
\end{proof}
\begin{corollary}
Under the same setting as in Lemma \ref{pointwise}, fixing $x$ as in the lemma, assume that some $h_0$ is bijective. Denoting by $|\cdot|_{h(x)}$ the norm induced by the scalar product $\langle \cdot, \cdot \rangle_{h(x)}$ introduced in the lemma, we have
\begin{equation}\label{asy}
\frac{|h(y)-h(x)|_{h(x)}}{\dist(y, x)} \to 1, \quad \text{for any $y \to x$.}
\end{equation}
\end{corollary}
\begin{proof}
Fix a sequence $y_i \to x$ with $y_i \neq x$ and let $r_i=\dist(x, y_i)$. Then (the proof of) Lemma \ref{pointwise} shows, after passing to a subsequence, under the same notations,
\begin{equation}
\frac{|h(y_i)-h(x)|_{h(x)}}{\dist(y_i, x)}=\frac{|h(y_i)-h(x)|_{h(x)}}{r_i} \to |h_0(y)|_{h(x)}
\end{equation}
for some $h_0$ and some $y \in \partial B_1(0_n)$ with respect to (\ref{pmgh}). Recalling the definition of $|\cdot |_{h(x)}$, $|h_0(y)|_{h(x)}$ must be equal to $|y|=1$. This completes the proof of (\ref{asy}).
\end{proof}
Finally let us give the following sufficient condition for the surjectivity of $h_0$. In the proof, we will use an well-known fact; for any $L$-Lipschitz map $f:Z \to W$ between metric spaces,
\begin{equation}\label{liphaus}
\haus^{\alpha}(f(A)) \le L^{\alpha}\haus^{\alpha}(A)
\end{equation}
holds for all $\alpha \ge 0$ and subset $A$ of $Z$. 
\begin{proposition}\label{surj}
Under the same setting as in Lemma \ref{pointwise}, if $(X, \dist,\meas)$ is a non-collapsed $\RCD(K, n)$ space and $\haus^n(h(A))>0$ holds for some Borel subset $A \subset X$ with $\haus^n(A)>0$, then for $\haus^n$-a.e. $x \in A \cap \mathcal{R}_n$, any $h_0$ is bijiective.
\end{proposition}
\begin{proof}
Let us fix a small $\epsilon>0$ and find $\{(A_i, \Phi_i)\}_{i \in \mathbb{N}}$ as in Theorem \ref{reduced}.
The local Lipschitz continuity of $h$ implies
\begin{equation*}
\haus^n \left( h(A)\setminus h\left(\bigcup_iA \cap A_i\right)\right) \le \haus^n\left( h\left(A \setminus \bigcup_iA_i\right)\right)=0.
\end{equation*}
Thus
\begin{equation*}
\haus^n\left(h\left(\bigcup_iA \cap A_i\right)\right)>0,
\end{equation*}
in particular, we have
\begin{equation*}
\haus^n\left(h(A \cap A_i)\right)>0
\end{equation*}
for some $i$ (thus $\haus^n(A_i \cap A)>0$). Applying a similar argument as in \cite{LN} to the map
\begin{equation*}
\Phi_i(A \cap A_i) \stackrel{\Phi_i^{-1}}{\to} A \cap A_i \stackrel{h}{\to} \mathbb{R}^n
\end{equation*}
with the area formula in the Euclidean setting, we obtain the desired conclusion. More precisely, denoting $\tilde \Phi_i:= h\circ \Phi_i^{-1}=(\tilde \phi_{i, 1}, \ldots, \tilde \phi_{i, n})$, we have
\begin{equation*}
\int_{\Phi_i(A \cap A_i)}\sqrt{\mathrm{det}(\langle \nabla \tilde \phi_{i, j}, \nabla \tilde \phi_{i, l}\rangle )_{jl}}\di \haus^n{\color{blue}\ge}\haus^n(h(A \cap A_i))>0.
\end{equation*}
Find $x \in \mathrm{Leb}(\Phi_i(A \cap A_i))$ such that $\tilde \Phi_i$ is differentiable at $x$ in the sense of Definition \ref{differen} and that $\mathrm{det}(\langle \nabla \tilde \phi_{i, j}, \nabla \tilde \phi_{i, l}\rangle )_{jl}(x)>0$. 
Thus for any $y \in \Phi_i(A \cap A_i)$
\begin{equation*}
\tilde \Phi_i(y)=\tilde \Phi_i(x)+J(\tilde \Phi_i)_x(y-x)+o(|y-x|).
\end{equation*}
Denoting $x=\Phi_i(z), y=\Phi_i(w)$, we have
\begin{equation*}
h(w)=h(z)+J(\tilde \Phi_i)_x(\Phi_i(w)-\Phi_i(z))+o(|\Phi_i(w)-\Phi_i(z)|).
\end{equation*}
Thus dividing by $\dist (w, z)$ in both sides above and then letting $w \to z$, we see that any $h_0$ is equal to $J(\tilde \Phi_i)_x$, up to an action by an element of $O(n)$ in the domain, whenever $\langle \nabla \phi_{i, j}, \nabla \phi_{i, l}\rangle(z)=\delta_{jl}$ holds which is valid after multiplying a matrix of size $n$ to $\Phi$ in the domain (see the proof of Lemma \ref{pointwise}). Since $J(\tilde \Phi_i)_x$ is invertible, we see that $h_0$ is bijective.
\end{proof}
\subsection{The case when $M$ is a smooth manifold}
We are now in a position to introduce the main result of Section \ref{deriv}.
\begin{corollary}\label{sss}
For $\meas$-a.e. $x \in \mathcal{R}_n=X$, there exists a unique linear map
\begin{equation*}
(Dh)_x:\mathbb{R}^n \to T_xM,
\end{equation*}
up to an action of an element in $O(n)$ in the domain, such that fixing a chart $(U, \Phi)$ around $h(x)$, for all blow-up limit maps
\begin{equation*}
(\Phi \circ h)_0=\lim_{r_i \to 0^+}(\Phi \circ h)_{r_i}: \mathbb{R}^n \to \mathbb{R}^n
\end{equation*}
and
\begin{equation*}
\tilde{(\Phi \circ h)}_0=\lim_{s_i \to 0^+}(\Phi \circ h)_{s_i}: \mathbb{R}^n \to \mathbb{R}^n
\end{equation*}
under (\ref{pmgh}),
we see that all $(\Phi \circ h)_0$,  $\tilde{(\Phi \circ h)}_0$  and $D(\Phi)_{h(x)} \circ (Dh)_x$ coincide, up to actions by elements of $O(n)$.
In particular if some $(Dh)_x$ is bijective, then we can define the push-forward scalar product $g_{h(x)}=\langle \cdot, \cdot \rangle_{h(x)}$ on $T_{h(x)}M$  by
\begin{equation*}
\langle v, w \rangle_{h(x)}:=(Dh)_x^{-1}(v)\cdot (Dh)_x^{-1}(w).
\end{equation*}
\end{corollary}
\begin{proof}
This is a direct consequence of Lemmas \ref{pointwise} and \ref{asyntt}.
\end{proof}
\section{Proof of Theorem \ref{LHRCD}}\label{secfinal}
Let $(X, \dist, \haus^n)$ be a locally homogeneous non-collapsed $\RCD(K, n)$ space and let $h:U \subset X \to M$ be a canonical homeomorphism discussed in \cite{LN}. We need to make use of some additional structure of $M$ described in the aforementioned reference. By construction $M$ is a local quotient $U_G/U_H$ of local Lie groups $U_G,U_H$. That means that $U_G$ acts by {\color{blue}local} isometries on $X$ and by smooth maps on $M$ {\color{blue}(thus this is not an action of a group in the usual sense)}. We endow $M$ with a left invariant Riemannian metric. For this fix some point $x_0$ (we assume without loss of generality that $h(x_0)=eH \in U_G/U_H = M$ and push{\color{blue}-}forward the metric onto $T_{h(x)}M$ as in Corollary \ref{sss}. Then using the smooth action of $U_G$ on $M$ we can define a left invariant (and hence smooth) Riemannian metric $g$ on the whole of $M$ by 
    $$ g_{f H}(v,w) := g_{eH}((DL)_{f^{-1}} (v) ,(DL)_{f^{-1}}(w)), $$ where $f$ denotes an element in $U_G$ and $fH \in M=U_G/U_H$. With the induced metric $\dist_g$ the local group $U_G$ acts by isometries on $M$. Now the map $h: U \to M$ being canonical means that we have $f \circ h \circ f^{-1} = h$. \par 
As  mentioned in the introduction the map $h\colon  (U,\dist) \rightarrow (M,\dist_g)$ is locally Lipschitz (since we could construct an auxiliary Riemannian metric $g'$ such that $h \colon (U, \dist) \rightarrow (M, \dist_{g'})$ is Lipschitz, compare \cite{LN}) and hence preserves the length of absolute continuous paths $\gamma: I \rightarrow U \subset X$.    
\begin{lemma}   \label{lem:isometry}
Let $h: U \subset X \rightarrow M$ be the canonical map described above and $\gamma: I \rightarrow U$ be an absolutely continuous curve. Then we have 
$$\int_I \vert \gamma'(t) \vert \,\di t = L(\gamma)= L(h \circ \gamma)=\int_I \vert (h \circ \gamma)'(t) \vert \, \di t,$$
where in the above equation under the integrals we consider the metric derivative (see \cite{AGS}), which exists for $\mathcal{L}^1$-a.e. $t \in I$ for absolutely continuous curves. 
    \end{lemma}
    \begin{proof}
        In \cite{LN} one could find some Riemannian metric $g'$ such that the map $h \colon (U,\dist) \rightarrow (M,g')$ is locally Lipschitz. Therefore $h$ takes absolutely continuous curves $\gamma$ to absolutely continous curves $h\circ \gamma$. Consider $t \in I$ such that the metric derivatives $\vert \gamma'(t) \vert , \vert (h\circ \gamma)'(t) \vert$ exist. Denote by $f \in U_G$ an element such that $f(\gamma(t))= x_0$ and $f^{-1}\circ h \circ f \circ \gamma(t)= h(\gamma(t))$, which is possible due to the canonicallity  of $h$. Then we have 
            $$ \vert \gamma'(t) \vert = \vert (f \circ \gamma)'(t) \vert = \vert (h \circ f \circ \gamma)'(t) \vert= \vert ( f^{-1} \circ h \circ f \circ \gamma)'(t) \vert.  $$
            Here the middle equations follows from the fact that $\vert (h \circ f \circ \gamma)'(t) \vert=\vert (h \circ f \circ \gamma)'(t) \vert_g$ if the metric is induced by a Riemannian metric and by definition of the pushforward at the point $x_0$. The other equalities follow from the fact,  that local isometries preserve the metric derivative. Since the arguments above is justified for $\mathcal{L}^1$-a.e. $t \in I$, the conclusion follows.
    \end{proof}
\begin{corollary}\label{cor3}
For all $x\in X$, $R>0$ and $\epsilon>0$, there exists $r_0>0$ such that
\begin{equation*}
\left| \frac{\dist_{g}(h(y),h(x))}{\dist(x, y)}-1\right|<\epsilon
\end{equation*}
for all $x, y \in B_R(x)$ with $\dist(x, y)\le r_0$. In particular $h^{-1}$ is locally Lipschitz.
\end{corollary}
\begin{proof}
By (\ref{asy}) we have for $x_0 \in U \subset X$
\begin{equation*}\label{asy2}
\frac{\dist_{g}(h(y), h(x_0))}{\dist(y, x_0)} \to 1, \quad \text{for any $y \to x_0$.}
\end{equation*}
Now, since $h$ is canoncial, the claim follows.
\end{proof}

\begin{corollary}
We see that $h$ is an isometry.
\end{corollary}
\begin{proof}
Lemma \ref{lem:isometry} shows that $h$ is a $1$-Lipschitz map. From Corollary \ref{cor3} we see  that $h^{-1}$ is locally Lipschitz and thus takes absolutely continuous paths to absolutely continuous paths. The same argument as in Lemma \ref{lem:isometry} implies that $h$ preserves the distances, namely
\begin{equation*}
\dist(x, y)=\dist_{g}(h(x), h(y)),\quad \text{for all $x, y \in X$.}
\end{equation*}
\end{proof}

\section{Proof of Corollary \ref{finite}}
Though the proof is quite standard, we give a proof for readers' convenience.
\begin{proof}[Proof of Corollary \ref{finite}]
We provide proofs of (1) and (3) only because a very similar argument allows us to get (2). 

Firstly let us prove (1) by a contradiction. If it is not the case, then there exists a sequence of non-collapsed $\RCD(K, n)$ spaces $(X_i, \dist_i, \haus^n)$ with $\mathrm{diam}X_i \le d$ and $\haus^n(X_i)\ge v$ such that 
\begin{equation}\label{almost hom2}
    \dist_{\mathrm{pGH}}\left((B_s(x_i),x_i), (B_s(y_i), y_i) \right) \le \epsilon_is
    \end{equation}
    for all $x_i, y_i \in X_i$ and $s \in (0, r]$,
and that $(X_i, \dist_i, \haus^n)$ is uniformly far from any locally homogeneous compact Riemannian manifold. After passing to a subsequence, the sequence mGH-converges to a non-collapsed compact $\RCD(K,n)$ space $(X, \dist, \haus^n)$.  Thanks to (\ref{almost hom2}), we know that $(X, \dist, \haus^n)$ is locally homogeneous. Thus by Theorem \ref{LHRCD}, it is smooth. Therefore applying \cite[Theorem A.1.2]{CheegerColding1} proves that $X_i$ is homeomorphic to $X$ for any sufficiently large $i$ (see also \cite[Theorem 1.2]{HP}). This is a contradiction.
    
    Finally let us prove (3). 
    The proof is also done by a contradiction. 
    
    Assume that the assertion is not satisfied. Then there exist $r>0$, a sequence $\epsilon_i \to 0^+$ and a sequence of Riemannian manifolds $(M_i, g_i)$ of dimension $n$ with $\mathrm{Ric}_{g_i} \ge K$, $\mathrm{diam}_{g_i}M_i\le d$ and  $\mathrm{vol}_{g_i}M_i \ge v$ such that $M_i$ is not diffeomorphic to $M_j$ for all $i \neq j$, and that 
    \begin{equation*}\label{almost hom}
    \dist_{\mathrm{pGH}}\left((B_s(x_i),x_i), (B_s(y_i), y_i) \right) \le \epsilon_is.
    \end{equation*}
    holds for all $x_i, y_i \in M_i$ and $s \in (0, r]$. Then as in the case of (1), after passing to a subsequence, $(M_i, \dist_{g_i}, \haus^n)$ mGH-converge to a locally homogeneous Riemannian manifold $(X, \dist, \haus^n)$. Therefore applying  \cite[Theorem A.1.12]{CheegerColding1}, we see that $M_i$ is diffeomorphic to $X$ for any sufficiently large $i$. This, in particular, implies that $M_i$ is diffeomorphic to $M_j$ for all sufficiently large $i, j$. Thus we have a contradiction.
\end{proof}
\begin{acknowledgement}
We would like to thank Christian Ketterer and Alexander Lytchak for helpful discussions. {\color{blue}They wish to thank the referee for valuable comments for a revision and are also grateful to Diego Corro for pointing out a reference \cite{Santos}.}
This note has been written for the special volume collecting the Proceedings of the meetings
“Analysis, Geometry and Stochastics on Metric Spaces” and “Metrics and Measures” that were held in RIMS and Tohoku on September 25-29, 2023.
The second named author is grateful to the Organizers of the meetings, for the kind invitation and the generous hospitality.
The second named author was partially supported by the DFG grant SPP 2026 project 24 and project 52.
The first named author acknowledges supports of the Grant-in-Aid for Scientific Research (B) of 20H01799, the
Grant-in-Aid for Scientific Research (B) of 21H00977 and Grant-in-Aid for Transformative
Research Areas (A) of 22H05105.
\end{acknowledgement}

\end{document}